\documentclass[letterpaper, 10 pt, conference]{ieeeconf}  

\IEEEoverridecommandlockouts                              

\usepackage{graphics} 
\usepackage{epsfig} 
\usepackage{amsmath} 
\usepackage{amssymb}  

\usepackage{amsfonts} 
\usepackage{mathdots}
\allowdisplaybreaks
\usepackage{stfloats}

\usepackage{xcolor}
\usepackage{cite}

\usepackage{latexsym,theorem}

{\theorembodyfont{\slshape}\newtheorem{theorem}{Theorem}[section]}
{\theorembodyfont{\slshape}}
{\theorembodyfont{\slshape}\newtheorem{lemma}[theorem]{Lemma}}
{\theorembodyfont{\slshape}}
{\theorembodyfont{\slshape}}
{\theorembodyfont{\upshape}}
{\theorembodyfont{\upshape}\newtheorem{problem}[theorem]{Problem}}
{\theorembodyfont{\upshape}\newtheorem{remark}[theorem]{Remark}}
{\theorembodyfont{\upshape}\newtheorem{assumption}[theorem]{Assumption}}
{\theorembodyfont{\upshape}}

\title{\LARGE \bf
Output-Feedback Controller Synthesis for Dissipativity and $H_2$ Performance  of Autoregressive Systems from Noisy Input-Output Data
}

\author{Pietro Kristović$^{1}$, Andrej Jokić$^{1}$ and Mircea Lazar$^{2}$
\thanks{$^{1}$ Faculty of Mechanical Engineering and Naval Architecture, University of Zagreb, Zagreb, Croatia, {\tt\small pietro.kristovic@fsb.unizg.hr, andrej.jokic@fsb.unizg.hr}}
\thanks{$^{2}$ Electrical Engineering Faculty, Eindhoven University of Technology, Eindhoven, The Netherlands
        {\tt\small m.lazar@tue.nl}}
\thanks{This research has been supported by the European Regional Development Fund under grant agreement PK.1.1.10.0007 (DATACROSS).}
}

\begin{document}

\maketitle
\thispagestyle{empty}
\pagestyle{empty}

\begin{abstract}
In this paper we propose a data-driven output-feedback controller synthesis method for discrete-time linear time-invariant  systems in the form of autoregressive model with exogenous input.
The synthesis goal is either to achieve dissipativity with respect to a given quadratic supply rate, or to achieve given $H_2$ performance level. It is assumed that the model of the plant is unknown, except for the disturbance term. To compensate for the lack of model knowledge, we have a recorded trajectory of the controlled input and the output available for control, which can be corrupted by an unknown but bounded disturbance. 
Derived controller synthesis method is in the form of linear matrix inequalities and is nonconservative within the considered problem setting.
\end{abstract}

\section{INTRODUCTION}
\label{IntroductionLabel}

In this paper, we propose a dynamic output-feedback controller synthesis method for discrete-time linear time-invariant (LTI) systems. We consider a class of LTI systems which can be expressed in the form of autoregressive (AR) model with exogenous input
\begin{equation}
\label{P_ARModel}
\begin{split}
& y(t)+A_1(y-1)+\cdots+A_l(y-l)= \\
& B_0u(t)+B_1u(t-1)+\cdots+B_lu(t-l)+B_ww(t),
\end{split}
\end{equation}
where $y$ is the measured output for control purposes, $u$ is the
control input, and $w$ is the disturbance input which affects the system in a specific way, more precisely, with no associated dynamics.
The goal of the synthesis is either to render the closed-loop system  dissipative with respect to a given quadratic supply function or to achieve a given $H_2$ performance level. 
The presented approach belongs to the class of direct data-driven approaches, where trajectory data is directly used in controller synthesis, bypassing a reconstruction of the system model.
More precisely, developed controller synthesis
follows the \textit{informativity approach} framework, which has gained a significant attention during past years, e.g. see
\cite{R27} for an overview.

Data-driven dissipativity analysis
has been considered in
\cite{996, 995, 994, 993}, while in  \cite{R26, c4, c6, R5564636, R24, zz1, zz66}  data-driven approaches were used to derive (stabilizing, $H_2$, $H_\infty$, quadratic performance) static state-feedback controller synthesis methods using input-state trajectories.
In contrast, 
the input-state data was used to derive dynamic output-feedback controller synthesis methods for dissipativity, $H_2$ and $H_\infty$ performance in \cite{999,66999}.
In this paper we are concerned with a dynamic output-feedback controller synthesis based on the input-output data. There are several approaches which deal within the same problem setting for deriving
(stabilizing, $H_2$, $H_\infty$, quadratic performance) controller synthesis  \cite{R26}, \cite{c5, a1, zz2, R23}. They all consider systems that belong to the same class of AR system \eqref{P_ARModel}.

The prevailing strategy in the existing solutions reported in the literature is to rewrite the AR model into a specific state-space form in which the state vector is composed of the set of time shifted control inputs and outputs available for control. In this way the knowledge of input-output data becomes the input-state data, while the controller synthesis problem is reduced to the static-state feedback problem, what allows one to use some of the existing state-feedback solutions.
However, it is recognized that the obtained synthesis problem in general cannot be simply solved by relying on this strategy \cite{c5, R33, R22}. There are two main challenges: 
\begin{itemize}
\item[\emph{i})] the corresponding state-space realization of an AR system is highly structured and if this is ignored the resulting solutions might be overly conservative;
\item[\emph{ii})] there are technical requirements on the recorded data (e.g. full rank condition) which are instrumental for employment of the so-called strict matrix S-lemma (or similar tools) on a path to devise a synthesis conditions for the \emph{set} of systems consistent with the data.
\end{itemize}
Some issues regarding \emph{(ii)} are in some more detail presented next.

For a given system in AR form \eqref{P_ARModel}, let $p$ denote the number of outputs available for control ($y\in \mathbb{R}^p$), and let $n$ and $l$ denote the order and the lag associated with a minimal state-space realization of the system, respectively. It turns out that the state-feedback solutions are applicable to the cases when $n=pl$. For convenience, in the remainder we will call such cases the \emph{restricted} cases while we refer to the condition $n=pl$ as the \emph{restriction}.

As shown in \cite{R33}, the above described restriction is directly related to uncontrollability of the space-state realization and rank deficiency of the associated state data matrix.
Within this context, we can distinguish:
\begin{itemize}
\item  methods in \cite{R26, c5}, which are restricted to the class of systems defined by relation $n=pl$.
\item  
methods in \cite{a1,zz2, R23}, in which the restriction $n=pl$ is bypassed by
assuming that the state\footnote{In \cite{R23} the behavioural framework is used 
and a state vector is not introduced, however full row rank condition is made on the data matrix which has the same structure as the state data matrix.}
 data matrix has full row rank.
In case $n=pl$, this rank condition can be obtained by using input which
is persistently exciting  of sufficiently high order.
In case $n<pl$, this rank condition can  be achieved only if the recorded state trajectories (the data) also span the complete uncontrollable subspace. The latter could in principle be achieved by using specific initial conditions.
\end{itemize}
Furthermore, the restriction to the class of systems for which $n=pl$ can be resolved by constructing special non-minimal realization using the data as shown in \cite{R33}.
In case signal-to-noise ratio is sufficiently large, newly obtained system realization can be used for  system stabilization based on state-feedback methods \cite{c6, R5564636}. However, it is unclear how to define and to interpret disturbance term in 
this new
realization of system dynamics and therefore adding input-output performance criteria, e.g., $H_\infty$ or $H_2$ (in addition to
stabilization) also remains unclear.
The restriction is also resolved in \cite{R22}, where system stabilization based on data corrupted by measurement error is considered. This is achieved by constructing an auxiliary system which increases the size of the state-space realization matrices, but ensures that the state data matrix has full row rank and $n=pl$.
However, in this approach it is not clear how to systematically construct the required auxiliary system, as recognized in \cite[Section~VI]{R33}.

Within the considered problem setting (AR model, dissipativity or $H_2$ performance), and
to the best of our knowledge, in this paper for the first time a dynamic output-feedback controller synthesis method is presented which simultaneously resolves the above presented challenges (\emph{i}) and (\emph{ii}). More precisely, regarding (\emph{i}), we fully exploit the structure, in contrast to \cite{R33} and \cite{R22}. With respect to (\emph{ii}), our approach is to construct a special data-driven system realization, similar to \cite{R33}. 
We exploit the fact that the state vector can remain in some (invariant) subspace (which might not be the complete state-space) even though the input is persistently exciting of 
an arbitrarily large
order. 
In this  way we obtain a new data-driven system realization by projecting 
 the non-minimal realization onto
the reachable subspace defined by the recorded data,  and thus ensure that the new state data matrix has a full row rank.
  The controller synthesis method is formulated in terms of linear matrix inequalities (LMIs) and is non-conservative,
i.e. LMIs represent necessary and sufficient conditions
for a set of systems defined by the recorded data and the known disturbance bound. 
The LMIs are more compact (smaller in size) compared to those in \cite{zz2}, and  are in form of strict inequalities which can be numerically verified. To illustrate results, we use a numerical example from \cite{R22} for which $n<pl$.

\subsubsection*{Notation} 
We use $\mathbb{R}$, $\mathbb{N}_{\geq 0}$, $\mathbb{R}^n$ and  $\mathbb{R}^{m\times n }$ to denote a field of real numbers, a set of natural numbers, an $n$-dimensional vectors with elements in $\mathbb{R}$, and $m$ by $n$ matrices with elements in $\mathbb{R}$, respectively.
For a real square matrix $M$ we use $\text{He}\{M\}$ to denote a symmetric matrix $M^\top + M$.
We use $M^\dagger$ to denote the Moore-Penrose pseudo-inverse of a real matrix $M$.
 We use $\text{im}(M)$, $\text{ker}(M)$, $\text{tr}(M)$ and $\text{rank}(M)$ to denote the image, kernel space, trace and rank of matrix $M$, respectively. We use $\text{diag}(A,B)$ to denote the matrix \begin{small}$\begin{pmatrix} A & 0 \\ 0 & B  \end{pmatrix}$\end{small}.  
In an LMI, $\star$ represents blocks which can be inferred from symmetry. We use  $\mathcal{C}(A,B)$ to denote the corresponding  controllability matrix, where $A$ and $B$ are the state and input matrix, respectively.
When replacing a matrix inequality of the form \begin{small}$\begin{pmatrix} A & B \\ B^\top & C \end{pmatrix}\prec 0$\end{small} with a set of inequalities $C \prec 0$, $A-BC^{-1}B^\top \prec 0$ (and vice versa), we will say that we have \emph{used the Schur complement with respect to the matrix $C$}. In such a way we explicitly specify which block diagonal matrix is inverted. We use $I_p$ and $0_r$ to denote identity matrix of size $p$ and square zero matrix of size $r$, respectively.
When we say that we apply a congruence transformation on
a matrix inequality $A\succ  0$ with respect to the matrix $S$, we
mean that $S$ is a full rank matrix, and that the transformed
inequality has the form $S^\top AS \succ 0$.
With matrices
$A_1, . . . , A_n$ which have the same number of columns, we
use $\text{col}(A_1, . . . , A_n)$ to denote the matrix $\begin{pmatrix} A_1^\top \cdots A_n^\top \end{pmatrix}^\top$.

\section{Preliminaries}
In this section we recall some definitions and results which will be instrumental in the remainder of the paper. 

\subsection{Dissipativity}
\label{SecDissipativity}
Consider a discrete-time LTI system 
\begin{equation}
\label{Sys1}
\begin{pmatrix} x(t+1) \\ z(t) \end{pmatrix} =
\begin{pmatrix} A & B \\ C & D \end{pmatrix}
\begin{pmatrix} x(t) \\  w(t) \end{pmatrix} 
\end{equation}
where $t \in \mathbb{N}_{\geq 0}$, $x(t) \in \mathbb{R}^n$, $w(t) \in \mathbb{R}^{m_w}$ and $z(t) \in \mathbb{R}^{p_z}$ represent  time, the system state, input and output, respectively. We say that the system \eqref{Sys1} is strictly dissipative with respect to a supply function $s(w(t),z(t))$ if there exists a storage function $V:\mathbb{R}^n \rightarrow \mathbb{R}$ and $\epsilon > 0$ such that the strict dissipation inequality
\begin{equation}
\label{Dissip1}
V(x(t))+s(w(t),z(t))- \epsilon \|\text{col}(x(t),w(t))\|^2 \geq V(x(t+1)) 
\end{equation}
holds for all $t \in \mathbb{N}_{\geq 0}$ and all trajectories $(w,x,z)$ of the system \eqref{Sys1}. With a quadratic storage function $V(x(t))=x(t)^\top P x(t)$ and a quadratic supply function
\begin{equation}
\label{supplyFunction}
s(w(t),z(t))=\begin{pmatrix} w(t) \\ z(t) \end{pmatrix}^\top \begin{pmatrix} -Q & -S \\ -S^\top & -R \end{pmatrix} \begin{pmatrix} w(t) \\ z(t) \end{pmatrix},
\end{equation}
the strict dissipation inequality \eqref{Dissip1} is equivalent to feasibility of the following LMI
\begin{equation}
\label{Dissip2}
\begin{pmatrix}
I & 0 \\
A & B \\
0 & I\\
C & D  \\
\end{pmatrix}^\top
\begin{pmatrix}
-P & 0 & 0 & 0  \\
0& P & 0 & 0   \\
0 & 0 & Q & S \\
0 & 0 & S^\top & R\\
\end{pmatrix}
\begin{pmatrix}
I & 0 \\
A & B \\
0 & I\\
C & D  \\
\end{pmatrix}
\prec 0.
\end{equation}
If in addition we have $P\succ 0$ and $R \succeq 0$, then \eqref{Dissip2} implies the stability of system \eqref{Sys1} since the Lyapunov inequality $A^\top P A - P \prec 0$ is incorporated in \eqref{Dissip2}. 

\begin{remark}
\label{RemarkPrvi}
The channel $w  \rightarrow z $ achieves $H_\infty$ performance of at least $\gamma $ (with $\gamma >0$) if and only if \eqref{Dissip2} holds for
$Q=-\gamma^2 I_m$,  $S = 0$, $R = I_{p_z}$ and some $P \succ 0$.
\end{remark}

\subsection{$H_2$ performance}
\label{H2performanceSubsection}
Consider a discrete time LTI system \eqref{Sys1}. Its transfer function matrix for the channel $w \rightarrow z$ is given by the equation
\begin{equation}
\nonumber
T(e^{j\omega})=C (e^{j\omega}I-A)^{-1}B+D,
\end{equation}
and its $H_2$ norm is denoted by $||T(e^{j\omega})||_2$.
According to \cite{zz00001}, $||T(e^{j\omega})||_2 < \mu $ if and only if  $\text{tr}(Z)< \mu^2$ and
\begin{equation}
\label{H2LMI}
\begin{pmatrix}
P-A P A^\top & B \\
B^\top &  I
\end{pmatrix}
 \succ 0, 
 \quad 
 \begin{pmatrix}
Z-DD^\top  & CP \\
P  C^\top &  P
\end{pmatrix} \succ 0.
\end{equation}
Note that \eqref{H2LMI} implies stability of the system \eqref{Sys1} since Lyapunov inequalities are incorporated in \eqref{H2LMI}.

\subsection{Matrix S-lemma}
The following version of the matrix S-lemma has been presented in \cite[Theorem 4.10]{c21}. 
\begin{theorem}
\label{Theorem1}
Let $M, H \in \mathbb{R}^{(p+r)\times (p+r)}$ be symmetric matrices and consider the partition \begin{small}$H=\begin{pmatrix} H_{11} & H_{12} \\ H_{12}^\top & H_{22} \end{pmatrix}$\end{small} where $H_{11}\in \mathbb{R}^{p \times p}$. Let the set $S_{H}$ be defined as
\begin{equation}
S_{H}:=\{ Z\in\mathbb{R}^{r\times p} | \begin{pmatrix} I \\ Z \end{pmatrix}^\top H \begin{pmatrix} I \\ Z \end{pmatrix} \succeq 0 \}.
\end{equation}
Assume\footnote{Note that in \cite[Theorem~4.10]{c21} it is in fact required that $H_{11}-H_{12}H_{22}^\dagger H_{12}^\top \succeq 0$ (i.e., the generalized Schur complement of $H$ with respect to $H_{22}$ is positive semidefinite). This condition is equivalent to the condition that $S_{H}$ is not an empty set, as shown  in \cite[page 6]{c21}.} that $H_{22}\prec 0$ and $S_{H} \neq \emptyset$.
Then, we have that
\begin{equation}
\begin{pmatrix}
I \\
Z
\end{pmatrix}^\top
M
\begin{pmatrix}
I \\
Z
\end{pmatrix}\succ 0 \text{ for all } Z\in S_{H}
\end{equation}
if and only if there exists a scalar $\alpha \geq 0$ such that
\begin{equation}
M-\alpha
H \succ
0. 
\end{equation}
\end{theorem}

\subsection{Dualization lemma}
The following lemma is commonly known as the dualization lemma \cite{R25}. 
\begin{lemma}
\label{Lemma2}  
Let $\Delta$ be a non-singular symmetric matrix in $\mathbb{R}^{n\times n}$, and let $\mathcal{U}$, $\mathcal{V}$ be 
two complementary subspaces whose sum equals  $\mathbb{R}^n$. Then
\begin{equation}
x^\top \Delta x<0 \hspace{0.7em} \forall x\in \mathcal{U}  \setminus \{0\} \text{\, and \,} x^\top \Delta x\geq 0 \hspace{0.7em} \forall x\in \mathcal{V}  \nonumber
\end{equation}
is equivalent to 
\begin{equation}
x^\top \Delta^{-1}x>0 \hspace{0.7em} \forall x\in \mathcal{U}^\perp  \setminus \{0\} \text{\, and  \,} x^\top \Delta^{-1}x\leq 0 \hspace{0.7em} \forall x\in \mathcal{V}^\perp, \nonumber
\end{equation}
where $\mathcal{U}^\perp$ and $\mathcal{V}^\perp$ are subspaces orthogonal to $\mathcal{U}$ and $\mathcal{V}$, respectively.
\end{lemma}

\section{Problem definition}
\label{Problem definition label}

\subsection{The plant, the controller and the state-space realization of the closed-loop system}
Consider the following class of discrete-time LTI system 
\begin{equation}
\label{ARM}
A(q^{-1})y(t)=B(q^{-1})u(t)+B_w w(t),
\end{equation}
where $y (t) \in \mathbb{R}^p$ is the measured output available for control, $u(t) \in \mathbb{R}^m$ is the control input, $w(t) \in \mathbb{R}^{m_w}$ is the disturbance input, and $q^{-1}$ represents the delay operator, i.e., $q^{-1}y(t)=y(t-1)$.
Let $n$ be the order of a minimal realization of the system \eqref{ARM}.
We assume that the matrix $B_w\in \mathbb{R}^{p\times m_w}$ is a full column rank matrix, and the matrices $A(\xi)\in\mathbb{R}^{p \times p} [\xi]$ and $B(\xi)\in\mathbb{R}^{p \times m} [\xi]$ are
matrices of polynomials in the indeterminate $\xi$, that is
\begin{equation}
\begin{split}
A(\xi)&:=I+A_1\xi+A_2\xi^2+\cdots+A_l\xi^l, \\
B(\xi)&:=B_0+B_1\xi+B_2\xi^2+\cdots+B_l\xi^l. \nonumber
\end{split}
\end{equation}

The system \eqref{ARM} can be represented in the following state-space form
\begin{subequations}
\label{eq:3:11}
\begin{align}
\chi (t+1)&=
A_z \chi(t)+B_z u(t)+\hat{B} w(t), \label{eq:3:11-A} \\ 
y(t) &= 
\begin{pmatrix}
 \bar{A} & \bar{B} 
 \end{pmatrix}
 \chi(t)+B_0u(t)+B_w w(t), \label{eq:3:11-B}
 \end{align}
\end{subequations} 
 where 
 \begin{equation}
 \label{ChiDefinition}
\chi(t):=\text{col} (y(t-1),..., y(t-l), u(t-1),..., u(t-l))
\end{equation} 
is the system state, and 
 \begin{equation}
 \label{ExtendedStateRealization}
 \begin{split}
A_z&:= \begin{pmatrix}
\bar{A} & \bar{B} \\ \hline
\begin{pmatrix} I_{p(l-1)} & 0 \end{pmatrix} & 0 \\
0 & 0 \\
 0 & \begin{pmatrix} I_{m(l-1)} & 0 \end{pmatrix}
 \end{pmatrix}:=
 \begin{pmatrix}
\begin{matrix} \bar{A} & \bar{B}  \end{matrix} \\  \hline
J_{A_z}
 \end{pmatrix}, 
  \\
 B_z &:=\begin{pmatrix}
B_0 \\ \hline 0 \\ I_m \\0
\end{pmatrix}:=
\begin{pmatrix}
B_0 \\  \hline
J_{B_z}
 \end{pmatrix},
 \quad
\hat{B}:=
\begin{pmatrix}
B_w \\ 0 \\ 0 \\ 0
\end{pmatrix}, \\
 \bar{A}&:=
 \begin{pmatrix}
 -A_1 & -A_2 & \cdots &  -A_l 
 \end{pmatrix}, \\
  \bar{B}&:=
 \begin{pmatrix}
 B_1 & B_2 & \cdots &  B_l 
 \end{pmatrix}.
 \end{split}
 \end{equation}
 
\begin{remark} 
\label{RemarkAR}
The system model \eqref{ARM} belongs to the class of  AR models, where $l$ denotes the system lag (observability index of a minimal realization of the system).
Furthermore, at time $t$ the disturbance $w(t)$
affects $y(t)$ directly by addition of $B_ww(t)$ (observe $I$
in $A(\xi)$), with no additional dynamics associated with $w$.
 The corresponding state space realization given by \eqref{eq:3:11} is characterized by structured state-space matrices \eqref{ExtendedStateRealization}. Exploiting this structure plays an important part in devising data-driven controller synthesis solutions.~$\hfill{\Box}$  
\end{remark} 
 
In connection to the system \eqref{ARM}, we further define the controlled output $z(t) \in \mathbb{R}^{p_z}$ as follows 
\begin{equation}
\label{eq:3:11-C}
 z(t)=C_z\chi (t)+D_zu(t)+\tilde{D}w(t), 
 \end{equation}
where $C_z \in \mathbb{R}^{p_z\times (p+m)l}$, $D_z \in \mathbb{R}^{p_z\times m}$ and $\tilde{D} \in \mathbb{R}^{p_z\times m_w}$. The complete open-loop system is therefore given by \eqref{ARM}, \eqref{ChiDefinition}, \eqref{eq:3:11-C}, in which the channel $u \rightarrow y$ is used for control, while the channel $w \rightarrow z$ is the \emph{performance channel}, on which we impose the desired closed loop specifications in terms of dissipativity or $H_2$ performance.

We consider a dynamic controller of the following form
\begin{equation}
\label{ControllerARM}
C(q^{-1})u(t)=D(q^{-1})y(t), 
\end{equation}
where, as before, $q^{-1}$ represents the delay operator, while the
matrices $C(\xi)\in\mathbb{R}^{m \times m} [\xi]$ and $D(\xi)\in\mathbb{R}^{m \times p} [\xi]$ are polynomial matrices
\begin{equation}
\begin{split}
C(\xi)&:=I+C_1\xi+C_2\xi^2+\cdots+C_l\xi^l, \\
D(\xi)&:=D_1\xi+D_2\xi^2+\cdots+D_l\xi^l. \nonumber
\end{split}
\end{equation}
The controller \eqref{ControllerARM} can be described using the following space-state form 
\begin{equation}
\begin{split}
\chi_c(t+1)&=A_c
 \chi_c(t) +
B_c
y(t),
\\
u(t) &= 
\begin{pmatrix}
 \bar{C} & \bar{D} 
 \end{pmatrix}
 \chi_c(t), \nonumber
\end{split}
\end{equation}
where 
\begin{equation}
\chi_c(t):=\text{col} (u(t-1),..., u(t-l), y(t-1),..., y(t-l)) \nonumber
\end{equation}
is the state vector, and
 \begin{equation}
 \begin{split}
 A_c&:=\begin{pmatrix}
\bar{C} & \bar{D} \\
\begin{pmatrix} I_{m(l-1)} & 0 \end{pmatrix} & 0 \\
0 & 0 \\
 0 & \begin{pmatrix} I_{p(l-1)} & 0 \end{pmatrix}
 \end{pmatrix}, \quad
 B_c:=
 \begin{pmatrix}
0 \\0 \\ I_p \\0
\end{pmatrix}, \\
 \bar{C}&:=
 \begin{pmatrix}
 -C_1 & -C_2 & \cdots & -C_l 
 \end{pmatrix}, \\
  \bar{D}&:=
 \begin{pmatrix}
 D_1 & D_2 & \cdots &  D_l 
 \end{pmatrix}. \nonumber
 \end{split}
 \end{equation}
Furthermore, note  that controller \eqref{ControllerARM} can be rewritten as 
\begin{equation}
\label{Controller}
u(t)= 
\begin{pmatrix}
 \bar{C} & \bar{D} 
 \end{pmatrix}
 \chi_c(t)=
 K\chi(t) ,
 \end{equation}
where $K=\begin{pmatrix}
 \bar{D} & \bar{C} 
 \end{pmatrix}$ and $\chi(t)$ is the one defined in \eqref{ChiDefinition}. Finally, by connecting the system \eqref{eq:3:11} in a feedback loop (on the control channel from $y$ to $u$) with the controller \eqref{Controller}, we obtain the closed-loop system 
\begin{equation}
\label{ClosedLoop}
\begin{split}
\chi(t+1)&=
\hat{A}\chi(t)+\hat{B}w(t), \\
z(t)&=\hat{C}\chi(t)+\tilde{D}w(t),
\end{split}
\end{equation}
where $\hat{A}:=A_z+B_zK$ and  $\hat{C}:=C_z+D_zK$.

\subsection{The unknown matrices and the known data}
We assume that we have no knowledge of  matrices $\bar{A}$, $\bar{B}$ and $B_0$, therefore, $A_z$ and $B_z$ are are not completely known, see \eqref{ExtendedStateRealization}. However, we assume that we know a finite time trajectory (of length $l+N $ time steps) of the control input $u$ and the corresponding output trajectory $y$. The corresponding disturbance input $w$ is not known. 
More precisely, the input-output data represented by matrices 
\begin{equation}
\label{RecordedData}
\begin{split}
Y & :=
\begin{pmatrix}
y(0) & y(1) & \cdots & y(N-1)
\end{pmatrix}), \\ 
X & :=
\begin{pmatrix}
\chi(0) & \chi(1) & \cdots & \chi(N-1)
\end{pmatrix}, \\
U & :=
\begin{pmatrix}
u(0) & u(1) & \cdots & u(N-1)
\end{pmatrix} \\
\end{split}
\end{equation}  
is known, while the disturbance data matrix 
\begin{equation}
\label{NoiseData}
W :=
\begin{pmatrix}
w(0) & w(1) & \cdots & w(N-1)
\end{pmatrix}
\end{equation}
is unknown.
Note that these data matrices satisfy  equation
\begin{equation}
\label{DataModel}
Y = \begin{pmatrix} \bar{A} & \bar{B} \end{pmatrix} X+B_0 U+B_wW.
\end{equation}

\subsection{Assumptions regarding the known data}
As presented in Section~\ref{IntroductionLabel}, all controller synthesis results rely on some constructed system realization for which the state data matrix has full row rank.
Along the same line, 
we introduce the following assumption.
\begin{assumption}
\label{Assumption-1}
$\chi(t)\in \text{im}(X)$ for all $t\in \mathbb{N}_{\geq 0}$ and any $u(t) \in \mathbb{R}^m$ and $w(t)\in \mathbb{R}^{m_w}$.~$\hfill{\Box}$
\end{assumption}
The above assumption is verifiable under the following conditions. Consider the following compact\footnote{The singular value matrix is a square matrix which contains only the non-zero singular values.} singular value decomposition 
\begin{equation}
\label{SVD}
X=X_{s}\Sigma_{\chi} X_{r}^\top, 
\end{equation}
and introduce the following abbreviation $X_d:=\Sigma_{\chi} X_{r}^\top$.
The matrices $X_{s} \in \mathbb{R}^{(p+m)l\times \tilde{n}}$ and $X_d \in \mathbb{R}^{\tilde{n}\times N} $ are a full column rank matrix and a full row rank matrix, respectively.
Furthermore, note that from \eqref{eq:3:11-A} it follows that 
 Assumption~\ref{Assumption-1} holds if and only if $\text{im}(\begin{pmatrix} A_z X_{s} & B_z &  \hat{B} \end{pmatrix} )\subseteq \text{im}(X_{s})$.
Therefore, we can define conditions under which Assumption~\ref{Assumption-1} holds using the persistency of excitation 
\cite[Corollary 2 (iii)]{c1} and  output representation \cite[Expression (4)]{R33}:
\begin{itemize}
\item[\emph{i})]
in the noiseless case, if the input $u$ is persistently exciting of order $n+l$, then Assumption~\ref{Assumption-1} holds if and only if $\text{im}(\hat{B}) \subseteq  \text{im}(X)$;
\item[\emph{ii})] in the noisy case, if the input $\text{col}(u,w)$ is persistently exciting  of order $n+l$, then Assumption~\ref{Assumption-1} holds. 
Although we do not control the input $w$,
the Assumption~\ref{Assumption-1} is achieved for almost any\footnote{The set of signals $w$ which would “cancel out” effects of $u$ on
the system in a way that results in reduction of $\text{im}(X)$ is of measure zero.} $w$ if the input $u$ is persistently exciting  of order $n+l$, and  the assumption is verified by checking whether $\text{im}(\hat{B}) \subseteq  \text{im}(X)$.
\end{itemize}
Note that the matrix $\hat{B}$, which appears in both \emph{i}) and \emph{ii}),  is assumed to be known. 
Thus, both \emph{i}) and \emph{ii}) can be verified after the data recording process. Additionally, the stated conditions can be used to design the data recording experiment. They also provide us with information when the data recording can be terminated.

Next, we make an assumption which will be instrumental for the construction of controller synthesis methods.
\begin{assumption}
\label{Assumption-2}
First $p$ rows of matrix $X$ are linearly independent.~$\hfill{\Box}$
\end{assumption}
The statement in the above assumption is necessarily satisfied when the output $y$ is controllable with respect to the input which is persistently exciting  of order $n+1$. This follows directly from output controllability matrix and the definition of order of persistency of excitation \cite{c1}.

\subsection{The disturbance model}
\label{The disturbance modelSubsection}
The disturbance affecting the data recording is assumed to be bounded. This is modeled using a quadratic matrix inequality imposed on $W$, as done in \cite{c6}. 
\begin{assumption}
\label{Assumption1}
The matrix $W$ from \eqref{NoiseData} satisfies the inequality  
\begin{equation}
\label{eq:2:6}
\begin{pmatrix}
I \\
W^\top\\
\end{pmatrix}^\top
\underbrace{
\begin{pmatrix}
\Phi_{11} & \Phi_{12}\\
\Phi_{12}^\top & \Phi_{22}
\end{pmatrix}}_\Phi
\begin{pmatrix}
I \\
W^\top\\
\end{pmatrix} \succeq 0,
\end{equation}
where $\Phi=\Phi^\top \in \mathbb{R}^{(p+N)\times(p+N)}$ is a known matrix with $\Phi_{22} \prec 0$. ~$\hfill{\Box}$   
\end{assumption}
Note that $\Phi_{22} \prec 0$ ensures that the set of matrices $W$ which satisfy \eqref{eq:2:6} is bounded.

\subsection{The data-driven realization of the closed-loop system}
\label{SynthesisConstructionTools}

Let $X_{s\perp}$  be a semi-orthogonal matrix whose columns form a basis of $\text{ker}(X_{s}^\top)$ and let
\begin{equation}
\nonumber
\chi(t)=
\begin{pmatrix}
X_{s} & X_{s\perp}
\end{pmatrix} 
\begin{pmatrix}
\chi_{s}(t) \\
\chi_{s\perp}(t)
\end{pmatrix}.
\end{equation}
If Assumption~\ref{Assumption-1} holds, then the closed-loop system \eqref{ClosedLoop} can be described using the following realization
\begin{equation}
\label{ReduciraniModel}
\begin{split}
\chi_s(t+1)&=
\tilde{A}\chi_s(t)+\tilde{B}w(t), \\
z(t)&=\tilde{C}\chi_s(t)+\tilde{D}w(t), \\
\end{split}
\end{equation}
where $\tilde{A}:=X_{s}^\top \hat{A} X_{s}$, $\tilde{B}:=X_{s}^\top \hat{B}$ and $\tilde{C}:=\hat{C}X_{s}$. 
Note that $\tilde{A} \in \mathbb{R}^{\tilde{n}\times \tilde{n}}$, where $\tilde{n}$ is the dimension of
the $\text{im}(X)$, see \eqref{SVD}, and that the system realization \eqref{ReduciraniModel} is
valid due to the invariance of the state subspace based on the Assumption~\ref{Assumption-1}.

\subsection{The set of plants consistent with data}
\label{SetOfSystemsSubsection}

Recall that the matrix $\begin{pmatrix} \bar{A} & \bar{B} \end{pmatrix}$ is an unknown matrix, thus, equation \eqref{DataModel} can be rewritten in the following manner
\begin{equation}
\label{DataModel-1}
Y = \begin{pmatrix} \bar{A}_s & \bar{B}_s \end{pmatrix} X_d+B_0U+B_wW,
\end{equation}
where matrix
\begin{equation}
\nonumber
\begin{pmatrix} \bar{A}_s & \bar{B}_s \end{pmatrix}:=\begin{pmatrix} \bar{A} & \bar{B} \end{pmatrix}X_{s}.
\end{equation}
Note that since $X_{s}$ has full column rank, matrix $\begin{pmatrix} \bar{A}_s & \bar{B}_s \end{pmatrix}$ can be arbitrary, i.e., it is not constrained
 (e.g., rank or structure constraints).
Let $\Sigma$ be the set of all triples ($\bar{A}_s$,$\bar{B}_s$,$B_0$) which can explain the data ($Y$, $X_d$, $U$), that is
\begin{equation}
\Sigma :=
\{(\bar{A}_s,\bar{B}_s,B_0) \,| \, \eqref{DataModel-1} \text{ holds for some }  W \text{satisfying } \eqref{eq:2:6} \}. \nonumber
\end{equation}
This set can be described using a quadratic matrix inequality which is suitable for application of matrix S-lemma.
\begin{lemma}
\label{Lemma3}
The set $\Sigma$ is the set of all triples ($\bar{A}_s$,$\bar{B}_s$,$B_0$) that satisfy the following inequality
\begin{equation}
\label{eq:2:3}
\begin{pmatrix}
I   \\ \hline
\bar{A}_s^\top  \\
\bar{B}_s^\top   \\
B_0^\top 
\end{pmatrix}^\top
\underbrace{
\begin{pmatrix}
H_{11}  & H_{12}  \\
H_{12}^\top & H_{22} 
 \end{pmatrix}}_{H}
\begin{pmatrix}
I   \\ \hline
\bar{A}_s^\top  \\
\bar{B}_s^\top   \\
B_0^\top 
\end{pmatrix} \succeq 0,
\end{equation}
where
\begin{equation}
\nonumber
H:=
\begin{pmatrix}
\star
\end{pmatrix}
\begin{pmatrix}
\Phi_{11}  & \Phi_{12}\\ 
\Phi_{12}^\top & \Phi_{22}
\end{pmatrix}
\begin{pmatrix}
B_w & Y  \\ \hline
0 &  -X_d \\
0 & -U
\end{pmatrix}^\top.
\end{equation}
\end{lemma}
\begin{proof}
If we multiply \eqref{eq:2:6} with matrices $B_w$ and $B_w^\top$ from the left and right side, respectively, then we obtain an equivalent matrix inequality since matrix $B_w$ has a full column rank. Next, if we substitute  $B_wW$  using \eqref{DataModel-1} we obtain \eqref{eq:2:3}.
\end{proof}

In order to finally satisfy all the assumptions of the matrix S-lemma, 
 we make the following assumption.
 \begin{assumption}
\label{Assumption-3}
The set $\Sigma$ is a nonempty set and we have that $\text{rank}(\text{col}(X_d,U))=\tilde{n}+m$.
~$\hfill{\Box}$
\end{assumption}
Recall that $\text{rank}(X_s)=\tilde{n}$, see \eqref{SVD}.
Thus, the rank condition from  Assumption~\ref{Assumption-3} can be achieved/verified in the same manner as done in connection to Assumption~\ref{Assumption-1}, with the difference that we need $n+l+1$ as the order of persistency of excitation instead of $n+l$.
Note that if Assumption~\ref{Assumption-3} holds, then $H_{22}\prec 0$ in \eqref{eq:2:3}, that is, the set $\Sigma$ is bounded \cite{c21}.

\subsection{The control problem}
\label{TheControlProblem}
In this paper, we are concerned with the following control problem. 
\begin{problem}
\label{Problem1}
Consider the system \eqref{ARM}, the controlled output \eqref{eq:3:11-C}, the controller \eqref{ControllerARM}, and suppose that we have the following knowledge: the recorded input-output data ($Y$, $X$, $U$) as defined in \eqref{RecordedData}; the disturbance bound  specified by $\Phi$ in Assumption~\ref{Assumption1}.
Let Assumptions~\ref{Assumption-1}, \ref{Assumption-2} and \ref{Assumption-3} hold.  
Recall that the considered system and controller have representations \eqref{eq:3:11} and \eqref{Controller}, respectively, and that they form a closed-loop system which can be described using system realization \eqref{ReduciraniModel}.
Consider two separate synthesis goals:
\begin{enumerate}
 \item[\emph{a})]
For a given performance supply function  \eqref{supplyFunction} with  regular matrix \begin{small}$ \begin{pmatrix} Q & S \\ S^\top & R \end{pmatrix} $\end{small} and  $R \succeq 0$,
design an  output-feedback controller \eqref{Controller} such that the closed-loop system \eqref{ReduciraniModel} is stable and strictly dissipative with respect to a given supply function for all triples $(\bar{A}_s, \bar{B}_s, B_0) \in \Sigma$.
 \item[\emph{b})] 
For a given $H_2$ performance $\mu$, 
design a output-feedback controller \eqref{Controller} such that the performance channel $w \rightarrow z$ of the closed-loop system \eqref{ReduciraniModel} achieves a given $H_2$ performance for all  triples $(\bar{A}_s, \bar{B}_s, B_0) \in \Sigma$.~$\hfill{\Box}$
\end{enumerate}   
\end{problem}

\begin{figure*}[b!]
 \par\noindent\rule{\textwidth}{0.5pt}
  \begin{equation}
  \tag{CS}
\label{ControllerSynthesis}
  \begin{gathered}
  \Pi_{11}:=
  \begin{pmatrix} 
-\alpha H_{22} & \begin{pmatrix} -\alpha H_{12}^\top & 0 \end{pmatrix}   \\
\star & L\tilde{P}L^\top+\begin{pmatrix} B_w\tilde{Q} B_w^\top-\alpha H_{11} & 0 \\ 0 & 0_{\tilde{n}-p} \end{pmatrix} & 
\end{pmatrix}, \quad
  \Pi_{12}:=
  \begin{pmatrix} 
 0  \\
\begin{pmatrix} B_w\tilde{Q} \tilde{D}^\top-B_w\tilde{S} \\ 0 \end{pmatrix} 
\end{pmatrix}, \\
  \Pi_{13}:=
\begin{pmatrix} 
 \begin{pmatrix} \tilde{P} \\ \tilde{K} \end{pmatrix} \\
F(J_{A_z}X_{s}\tilde{P}+J_{B_z}\tilde{K}) \\ 
\end{pmatrix}, \quad
\Pi_{22}:=
\tilde{D}\tilde{Q}\tilde{D}^\top -\text{He}(\tilde{D} \tilde{S})+\tilde{R}, 
\quad 
\Pi_{23}:=  C_zX_s\tilde{P}+D_z\tilde{K}.
\end{gathered}
\end{equation}
\end{figure*}

\section{Controller synthesis}
\label{Controller synthesis Label}

The following definitions are instrumental for solving the Problem~\ref{Problem1}.
Consider a matrix partition $X_s=\text{col}(X_{s1}, X_{s2})$ where $X_{s1}\in \mathbb{R}^{p\times \tilde{n}}$. Then, if the Assumption~\ref{Assumption-2} holds, the equation
\begin{equation}
\label{ipakTreba1}
LX_{s}^\top=
\begin{pmatrix}
\begin{pmatrix} I_p \\ 0 \end{pmatrix} & F
\end{pmatrix}
\end{equation}
has a solution 
\begin{equation}
\label{Definitions2}
L:=
\begin{pmatrix}
L_1 \\ L_2
\end{pmatrix}, \quad F:= \begin{pmatrix} F_{12} \\ F_{22} \end{pmatrix}=\begin{pmatrix} L_1 X_{s2}^\top \\ L_2 X_{s2}^\top  \end{pmatrix},
\end{equation}
where $L_1=(X_{s1} X_{s1}^\top)^{-1} X_{s1}$ and $L_2^\top$ is an arbitrary full column matrix whose columns form a basis of $\text{ker}(X_{s1})$. 
The following theorems present a solution to the Problem~\ref{Problem1}.
Note that statements in these theorems are in accordance with  dissipativity and $H_2$ performance description presented in Sections~\ref{SecDissipativity} and \ref{H2performanceSubsection}, respectively.

\begin{theorem}
\label{Theorem2}
Consider the Problem~\ref{Problem1}a and definitions of matrices in expression \eqref{ControllerSynthesis} presented 
at the bottom of the page. Then, the following two statements are equivalent:
\begin{enumerate}
 \item[\emph{i})] Matrix inequalities
 \begin{subequations}
 \label{eq:3:4}
 \begin{align}
 & P\succ 0, \quad R\succeq 0, \\
 & \begin{pmatrix}
\star
\end{pmatrix}^\top
\underbrace{
\begin{pmatrix}
-P & 0 & 0 & 0  \\
0& P & 0 & 0   \\
0 & 0 & Q & S \\
0 & 0 & S^\top & R\\
\end{pmatrix}}_{\Delta }
\underbrace{\begin{pmatrix}
I & 0 \\
\tilde{A} & \tilde{B} \\
0 & I\\
\tilde{C} & \tilde{D}  \\
\end{pmatrix}}_U
\prec 0,
\end{align}
 \end{subequations}
hold for all $(\bar{A}_s,\bar{B}_s,B_0)\in \Sigma$.
\item[\emph{ii})]
Matrix inequality 
\begin{equation}
\label{LMI_Theorem2}
\begin{pmatrix}
\Pi_{11} & \Pi_{12} & \Pi_{13} \\
\Pi_{12}^\top & \Pi_{22} & \Pi_{23} \\
\Pi_{13}^\top & \Pi_{23}^\top & \tilde{P}
\end{pmatrix} \succ 0,
\end{equation}
hold, where $\tilde{Q}\succeq 0$, $L$ and $F$ are defined in \eqref{Definitions2} and
\begin{equation}
\begin{gathered}
\label{InverseDefinitions}
\tilde{P}:=P^{-1}, \,
\begin{pmatrix}
\tilde{Q} & \tilde{S} \\
\tilde{S}^\top & \tilde{R}
\end{pmatrix}:=
\begin{pmatrix}
Q & S\\
S^\top & R
\end{pmatrix}^{-1}, \\
 \tilde{K}:=KX_s\tilde{P}.
\end{gathered}
\end{equation}
\end{enumerate}
\end{theorem}

\emph{Proof of Theorem~\ref{Theorem2}:}
Suppose the statement $\emph{i})$ is true.
Note the definition of matrices $\Delta$ and $U$ in \eqref{eq:3:4}, and consider the definitions in \eqref{InverseDefinitions}.
Let $\mathcal{U}=\text{im}(U)$, $\mathcal{V}=\text{im}(V)$, where \begin{small}$V:=\begin{pmatrix} 0 & I & 0 & 0 \\ 0 & 0 & 0 & I \end{pmatrix}^\top$\end{small}
such that $V^\top \Delta V = \text{diag}(P,R)$. Then,  according to Lemma~\ref{Lemma2}, \eqref{eq:3:4} is equivalent to
\begin{subequations}
\label{eq:3:2}
\begin{align}
& \tilde{P} \succ 0,  \label{eq:3:2A}\\
& \tilde{Q} \preceq 0,  \label{eq:3:2A1} \\
& \begin{pmatrix}
\star
\end{pmatrix}^\top
\underbrace{
\begin{pmatrix}
-\tilde{P} & 0 & 0 & 0  \\
0& \tilde{P} & 0 & 0   \\
0 & 0 & \tilde{Q} & \tilde{S} \\
0 & 0 &\tilde{S}^\top & \tilde{R}\\
\end{pmatrix}}_{\Delta^{-1} }
\begin{pmatrix}
\tilde{A}^\top & \tilde{C}^\top \\
-I & 0 \\
\tilde{B}^\top & \tilde{D}^\top \\
0 & -I  \\
\end{pmatrix}
\succ 0 \label{eq:3:2B}.
\end{align}
\end{subequations}
We can apply Schur complement rule on \eqref{eq:3:2A} and \eqref{eq:3:2B} with respect to $\tilde{P}$, to obtain the following inequality, which is equivalent to \eqref{eq:3:2}:
\begin{subequations}
\label{Novo}
\begin{align}
&\tilde{Q} \preceq 0, \label{NovoA} \\
&
\begin{pmatrix}
\tilde{P}+\tilde{B}\tilde{Q}\tilde{B}^\top & \tilde{B}\tilde{Q}\tilde{D}^\top-\tilde{B}\tilde{S} & \tilde{A} \\
\tilde{D}\tilde{Q}\tilde{B}^\top-\tilde{S}^\top\tilde{B}^\top & \tilde{D}\tilde{Q}\tilde{D}^\top-\text{He}\{ \tilde{D}\tilde{S} \}+\tilde{R} & \tilde{C} \\
\tilde{A}^\top & \tilde{C}^\top  & P
\end{pmatrix} \succ 0. \label{NovoB}
\end{align}
\end{subequations}
By applying the congruence transformation on \eqref{NovoB} with respect to the matrix
 $S_0 := \text{diag}(I_{\tilde{n}+p_z},\tilde{P})$
we obtain an equivalent expression
\begin{equation}
\label{eq:3:9}
\underbrace{
\begin{pmatrix}
\tilde{P}+\tilde{B}\tilde{Q}\tilde{B}^\top & \tilde{B}\tilde{Q}\tilde{D}^\top-\tilde{B}\tilde{S} & \tilde{A}\tilde{P} \\
\tilde{D}\tilde{Q}\tilde{B}^\top-\tilde{S}^\top\tilde{B}^\top & \tilde{D}\tilde{Q}\tilde{D}^\top-\text{He}\{ \tilde{D}\tilde{S} \}+\tilde{R} & \tilde{C}\tilde{P} \\
\tilde{P}\tilde{A}^\top & \tilde{P}\tilde{C}^\top  & \tilde{P}
\end{pmatrix}}_{\bar{\Theta}} \succ 0.
\end{equation}
Next, we apply the congruence transformation on \eqref{eq:3:9} with respect
to matrix $S_1 := \text{diag}(L^\top,I_{p_z+\tilde{n}})$ in order to obtain the
equivalent inequality $\Theta:=S_1^\top \bar{\Theta}S_1 \succ 0$.
Let matrix $\Theta$ be partioned in the following manner
\begin{equation}
\nonumber
\Theta=:\begin{pmatrix}
\Theta_{11} & \Theta_{12} \\
\Theta_{12}^\top & \Theta_{22}
\end{pmatrix},
\end{equation}
where $\Theta_{11}\in \mathbb{R}^{p \times p}$. Then, by using Schur complement rule on $\Theta$ with respect to $\Theta_{22}$, we get an equivalent expression
\begin{subequations} 
\label{eq:3:3}
\begin{align}
& \label{eq:3:3A}
\Gamma
\underbrace{(\begin{pmatrix} \Theta_{11} & 0 \\ 0 & 0 \end{pmatrix}- \tilde{\Gamma} \Theta_{22}^{-1} \tilde{\Gamma}^\top  )}_{M}
\Gamma^\top
\succ 0,  \\ 
 & \label{eq:3:3B} \Theta_{22} \succ 0 ,
\end{align}
\end{subequations}
where $\Theta_{12}=\Gamma \tilde{\Gamma}$  with $\tilde{K}:=KX_{s}\tilde{P}$ and
\begin{equation}
\nonumber
\Gamma:=
\begin{pmatrix}
I_p & \begin{pmatrix} \bar{A}_s & \bar{B}_s \end{pmatrix} & B_0
\end{pmatrix}, \quad 
\tilde{\Gamma}:=
\begin{pmatrix}
\tilde{\Gamma}_1 & \tilde{\Gamma}_2
\end{pmatrix},
\end{equation}
\begin{equation}
\nonumber
\tilde{\Gamma}_1:=
\begin{pmatrix}
\begin{pmatrix} I_p & 0 \end{pmatrix} L \tilde{P} L^\top  \begin{pmatrix} 0 \\ I_{\tilde{n}-p}  \end{pmatrix} 
&  B_w \tilde{Q} \tilde{D}^\top -B_w \tilde{S} \\
0  & 0 \\
0 & 0
\end{pmatrix},
\end{equation}
\begin{equation}
\nonumber
\tilde{\Gamma}_2:=
\begin{pmatrix}
F_{12}J_{A_z}X_{s}\tilde{P}+F_{12}J_{B_z}\tilde{K} \\
\tilde{P}   \\
 \tilde{K}
\end{pmatrix}.
\end{equation}
Note that $\tilde{K} \in \mathbb{R}^{m\times \tilde{n}}$ is an arbitrary matrix since matrices $\tilde{P}$ and $X_s$ are full rank matrix and full column rank matrix, respectively.
Matrix inequalities \eqref{eq:3:3A} and \eqref{eq:2:3} are in a form suitable for the application of Theorem~\ref{Theorem1} since Assumption~\ref{Assumption-3} holds.
Therefore, we conclude that we can use Theorem~\ref{Theorem1} to state that \eqref{eq:3:3A} holds for all triples $(\bar{A}_s,\bar{B}_s,B_0)$ which satisfy the inequality  (\ref{eq:2:3}), if and only if there exist $\alpha \geq 0$ such that
\begin{equation}
\label{s-lemmaApplied}
\begin{pmatrix} \Theta_{11} & 0 \\ 0 & 0 \end{pmatrix}- \tilde{\Gamma} \Theta_{22}^{-1} \tilde{\Gamma}^\top  -\alpha H \succ 0.
\end{equation}
Since \eqref{eq:3:3B} and \eqref{s-lemmaApplied} hold,  we can apply the Schur complement rule on \eqref{s-lemmaApplied} with respect to $\Theta_{22}$ to obtain the following equation
\begin{equation}
\nonumber
\begin{pmatrix}
\begin{pmatrix} \Theta_{11} & 0 \\ 0 & 0 \end{pmatrix}-\alpha H & \tilde{\Gamma} \\
\tilde{\Gamma}^\top & \Theta_{22}
\end{pmatrix}\succ 0.
\end{equation}
Finally, if we apply congruence transformation with respect to matrix $S_2:=\text{diag}(\begin{pmatrix} 0 & I_p \\ I_{\tilde{n}+m} & 0 \end{pmatrix}, I)$ we obtain \eqref{LMI_Theorem2}.
~$\hfill{\blacksquare}$
\begin{theorem}
\label{Theorem3}
Consider the Problem~\ref{Problem1}b and definitions of matrices in expression \eqref{ControllerSynthesis}. Then, the following two statements are equivalent:
\begin{enumerate}
 \item[\emph{i})] Matrix inequalities
\begin{equation}
\label{H2LMIsynth}
\begin{pmatrix}
\tilde{P}-\tilde{A} \tilde{P} \tilde{A}^\top & \tilde{B} \\
\tilde{B}^\top &  I
\end{pmatrix}
 \succ 0, 
 \quad 
 \begin{pmatrix}
Z-\tilde{D}\tilde{D}^\top  & \tilde{C} \tilde{P} \\
\tilde{P}  \tilde{C}^\top &  \tilde{P}
\end{pmatrix} \succ 0,
\end{equation}
with $\text{tr}(Z)< \mu^2$ hold for all $(\bar{A}_s,\bar{B}_s,B_0)\in \Sigma$.
\item[\emph{ii})]
Matrix inequalities 
 \begin{subequations}
 \label{eq:33:4}
 \begin{align}
\label{LMI_Theorem3a}
&
\begin{pmatrix}
\Pi_{11}  & \Pi_{13} \\
\Pi_{13}^\top  & \tilde{P}
\end{pmatrix} \succ 0,
\\ &
\label{LMI_Theorem3b}
 \begin{pmatrix}
Z-\tilde{D}\tilde{D}^\top & C_zX_s\tilde{P}+D_z\tilde{K} \\
\star &  \tilde{P}
\end{pmatrix} \succ 0,
 \end{align}
\end{subequations}
with $\text{tr}(Z)< \mu^2$  hold, where $L$ and $F$ are defined in \eqref{Definitions2}, $\tilde{P}:=P^{-1}$ and $\tilde{K}:=KX_s\tilde{P}$.
\end{enumerate}
\end{theorem}
\emph{Proof of Theorem~\ref{Theorem3}:} Suppose the statement $\emph{i})$ is true. If we apply Schur complement rule with respect to $\tilde{P}$ and $I$ on first matrix inequality in \eqref{H2LMIsynth}, and then apply congruence transformation with respect to matrix $S_3:=\text{diag}(I,\tilde{P})$ we can obtain equivalent matrix inequality
\begin{equation}
\nonumber
\begin{pmatrix}
\tilde{P}-\tilde{B}\tilde{B}^\top & \tilde{A}\tilde{P} \\
\tilde{P}\tilde{A}^\top  & \tilde{P}
\end{pmatrix}\succ 0.
\end{equation}
Note that this matrix inequality is equal to \eqref{eq:3:9} if we exclude block matrices from second row and column, with $\tilde{Q}=-I$. Thus, we can follow same the same argumentation as in the proof of Theorem~\ref{Theorem2} and excluding: matrix $I_{p_z}$ from definition of $S_1$; block matrices from second column of $\tilde{\Gamma}_1$. By following this procedure we obtain matrix inequality \eqref{LMI_Theorem3a}. Finally, the second inequality in \eqref{H2LMIsynth} is equivalent to \eqref{LMI_Theorem3b}, this finishes the proof.~$\hfill{\blacksquare}$

\begin{remark}
After solving synthesis LMIs numerically, the controller $K$ that solves Problem~\ref{Problem1} can be determined via matrix linear equation $K X_s=\tilde{K}\tilde{P}^{-1}$. 
Note that in \eqref{ARM} we can also use $l$ which is larger than the real system lag, i.e. we only need a upper bound of the lag, but this might introduce conservatism due to inclusion of unnecessary model parameters. 
\end{remark}

\section{Example}
\label{ExampleSection}


To illustrate the results we use an unstable discrete-time system from \cite{R22}, where it was used to illustrate the controller synthesis for stabilization.
To add the performance related channel, we extend the system with a performance input-output pair.
The matrices which define the AR model are 
\begin{equation}
\nonumber
\bar{A} =
\begin{pmatrix}
0 & 1 & 0  & 0 \\
0 & 1 & 1 & -1
\end{pmatrix},
\quad 
\bar{B} =
\begin{pmatrix}
2 & 0 & 0 & 0 \\
1 & 1 & -1 & -1
\end{pmatrix},
\end{equation}
and $B_0=0$.  These matrices are obtained, following the procedure from \cite[Expression (4b)]{R22}, from a minimal state-space realization
\begin{equation}
\nonumber
\begin{pmatrix}
x(t+1) \\ y(t)
\end{pmatrix}
=
\begin{pmatrix}
A_m & B_m \\
C_m & D_m
\end{pmatrix}
\begin{pmatrix}
x(t) \\ u(t)
\end{pmatrix},
\end{equation}
where
\begin{equation}
\nonumber
\begin{pmatrix}
A_m & B_m \\
C_m & D_m
\end{pmatrix}=
\begin{pmatrix}
\begin{array}{ccc|cc}
0 & 1 & 0 & 1 & 0 \\
-1 & 0 & 0 & 0 & 1 \\
0 & 0 & 1 & 1 & 0 \\ \hline
1 & 0 & 1 & 0 & 0  \\
0 & 1 & 1 & 0 & 0
\end{array}
\end{pmatrix}.
\end{equation}
Note that for the considered system $n=3$, $p=2$, $l=2$, thus, it holds that $n<pl$.

To the AR model we add the disturbance input $w(t) \in \mathbb{R}$ associated with matrix
\begin{equation}
\nonumber
B_w=
\begin{pmatrix}
0 \\
1
\end{pmatrix},
\end{equation}
for which it holds that  $\text{im}(\mathcal{C}(A_z,\hat{B}))\subseteq \text{im}(\mathcal{C}(A_z,B_z))$. In this case, the state $\chi(t)$ cannot span the whole space, and we can use persistently exciting  input $u$ to achieve Assumption~\ref{Assumption-1}.

The considered performance output $z(t)$ is defined by the following matrices
\begin{equation}
\nonumber
C_z:= \begin{pmatrix} 1 & 0 & 0 & 0 & 0 & 0 & 0 & 0\end{pmatrix}
,\quad 
D_z=0, \quad  
\tilde{D}= 
-1.  \nonumber
\end{equation}

This system model is used to generate exact and noisy input-output data.
The disturbance bound and the method for generating the noisy input-output data are the same as in \cite{c6}.
Input samples and initial condition $\chi(0) \in \mathcal{C}(A_z,\begin{pmatrix} B_z & \hat{B} \end{pmatrix}) $ are generated using a Gaussian distribution with zero mean and unit variance. Disturbance samples are generated in the same manner, but with standard deviation $ \sigma \in (0, 0.01, 0.05, 0.1, 0.2)$ for exact case and noisy cases, respectively.
The disturbance bounds have the following form
\begin{equation}
W_-W_-^\top\preceq 1.35 N \sigma^2 I, \nonumber
\end{equation}
where the matrix $W_-$ is defined as in \eqref{NoiseData}, and $N=32$ represents number of data samples (as in \cite{R22}).
Assumptions~\ref{Assumption-2}, \ref{Assumption1} and \ref{Assumption-3} regarding the data matrices and the disturbance bound were verified. 
For controller synthesis we consider minimizations of $H_\infty$ performance $\gamma$ (as defined in Remark~\ref{RemarkPrvi}) and $H_2$ performance $\mu$.
To solve the synthesis LMIs along with minimization of  $\gamma$ and $\mu$ we used Yalmip \cite{c22} environment in MATLAB, with Mosek as a LMI solver.

 The $H_\infty$ performance $\gamma$ results are presented in Figure~\ref{fig-1}, where by the term frequency response we refer to the largest singular value of the
corresponding transfer function on a particular
frequency $\omega $.  
Numerical results for $H_\infty$ and $H_2$ performance are presented in Table~\ref{Tablica1}. Note that closed-loop performances using model-based and data-based controllers obtained with $\sigma=0$ are the same. Furthermore, the performance of the closed-loop systems gets worse with an increase of noise in the recorded data,  which is to be expected since set of systems consistent with data increases.

\begin{figure}[h!]
\centering
\vspace{-0.5cm}
\includegraphics{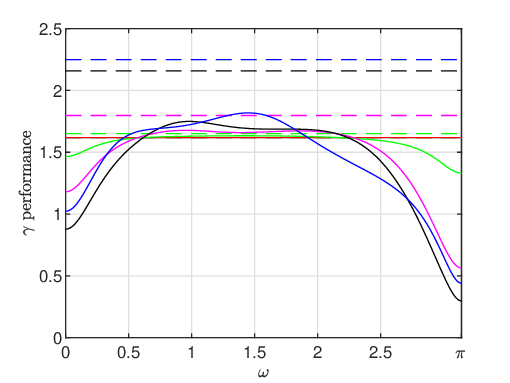}
\caption{Frequency responses (solid lines) of closed-loop system with data-based controller, and associated $H_\infty$ performance bounds (dashed lines) for $\sigma \in (0,0.01, 0.05, 0.1, 0.2)$  in red, green, magenta, black and blue color, respectively.}
\label{fig-1}
\end{figure} 

\begin{table}[h!]
\renewcommand{\arraystretch}{1.3}
\caption{Numerical results for $H_\infty$ and $H_2$ control.}
\begin{center}
\label{Tablica1}
\begin{tabular}{ |c|c|c|c|c| } 
\hline
$\sigma$  & $\gamma$ & $\gamma$ bound  & $\mu$ & $\mu$ bound\\
\hline
0  & 1.618 & 1.618 & 1.414 & 1.414 \\ 
0.01  & 1.633 & 1.649  & 1.414 & 1.431\\ 
0.05   & 1.677 & 1.797 & 1.416 & 1.464 \\ 
0.1   & 1.749  & 2.158 & 1.422 & 1.565	 \\
0.2  & 1.818  & 2.249 & 1.473 & 2.013 \\
\hline
\end{tabular}
\end{center}
\end{table}

\section{CONCLUSIONS}
In this paper we have proposed a novel non-conservative data-driven dynamic output-feedback controller synthesis method for the class of discrete-time LTI systems with closed-loop performance criteria formalized in terms of dissipativity or $H_2$ performance.
In contrast to related works, we overcome common restriction $n=pl$ and completely exploit the structure of constructed closed-loop system realization for the purpose of controller synthesis. 
The presented numerical example illustrates the effectiveness of the proposed method.

\addtolength{\textheight}{-12cm}   


\bibliographystyle{IEEEtran} 
\bibliography{references}

\end{document}